\documentclass{article}
\newcommand*{\arXiv}[2]{#1} %for arXiv

\arXiv{\usepackage{maa-monthly-my} \final}{\usepackage{maa-monthly} \final}
\usepackage{hyperref}

\usepackage{enumerate}
\usepackage{epsfig,lpic,wrapfig}
\usepackage{pifont}
\usepackage{caption}
\usepackage{wrapfig}

\newcommand*{\z}[1]{#1\nobreak\discretionary{}%
            {\hbox{$\mathsurround=0pt #1$}}{}}
            
\theoremstyle{theorem}
\newtheorem{theorem}{Theorem}
\newtheorem*{maintheorem}{\arXiv{Main theorem}{Main Theorem}}
\newtheorem{Proposition}[theorem]{Proposition}
\newtheorem{Crofton-type formula}[theorem]{Crofton-type formula}
\newtheorem{Lemma}[theorem]{Lemma}
\newtheorem{Douglas--Rado theorem}[theorem]{\arXiv{Douglas--Rado theorem}{Theorem}}
\newtheorem{Extended monotonicity theorem}[theorem]{\arXiv{Extended monotonicity theorem}{Theorem}}
\newtheorem{Theorem}[theorem]{Theorem}
\newtheorem{Claim}[theorem]{Claim}
\newtheorem{Observation}[theorem]{Observation}

\theoremstyle{definition}

\def\area{\mathop{\rm area}\nolimits}
\def\length{\mathop{\rm length}\nolimits}
\def\cross{\mathop{\rm cross}\nolimits}
\def\DD{\mathbb{D}}
\def\RR{\mathbb{R}}
\def\ZZ{\mathbb{Z}}

\def\solidtriangle{\text{\ding{115}}}
\newcommand*{\tc}[1]{\ifthenelse{\equal{#1}{}}{\Phi}{\Phi({#1})}}%toatal curvature \tc\gamma produces \Phi(\gamma)

\def\epsilon{\varepsilon}
\def\eps{\varepsilon}
\def\phi{\varphi}
\def\emptyset{\varnothing}

\def\ge{\geqslant}
\def\le{\leqslant}

\newcommand{\threestars}{\smash{%
\raisebox{-.6ex}{%
\setlength{\tabcolsep}{.5pt}%
\begin{tabular}{@{}cc@{}}%
\multicolumn2c*\\[-1.7ex]*&*%
\end{tabular}}}}

\newcommand{\fourstars}{\smash{%
\raisebox{-.6ex}{%
\setlength{\tabcolsep}{.5pt}%
\begin{tabular}{@{}cc@{}}%
*&*\\[-1.7ex]*&*%
\end{tabular}}}}

\hypersetup{
colorlinks=true,citecolor=black,linkcolor=black,anchorcolor=black,filecolor=black,menucolor=black,urlcolor=black,
pdftitle={Six proofs of the Fáry--Milnor theorem},
pdfauthor={Anton Petrunin and Stephan Stadler}
}

\begin{document}

\title{
\arXiv{Six proofs of the Fáry--Milnor theorem}{Six Proofs of the Fáry--Milnor Theorem}
}
\author{Anton Petrunin and Stephan Stadler}
\date{}
\maketitle

\begin{abstract}
We survey known proofs of the Fáry--Milnor theorem;
it states that any nontrivial knot makes at least two full turns.
\end{abstract}

\section*{INTRODUCTION.}

The following problem was posed by Karol Borsuk \cite{borsuk}.

\smallskip

\textit{Show that the total curvature of any nontrivial knot is at least $4\arXiv{{\cdot}}{}\pi$.}

\smallskip

It is known by many proofs based on different ideas.
We sketch several solutions, one solution per section;
each can be read independently.

This problem also has a number of refinements and generalizations;
in particular, a strict inequality holds --- this is the famous \emph{Fáry--Milnor theorem}.
However, for the sake of simplicity, we stick to the original formulation.

In order to continue, we need to agree on a definition of knot and explain what a \emph{nontrivial knot} is.
Despite the intuitive idea of a knot as a \emph{cyclic rope},
the simplest formal definition uses polygonal curves.
It turns out that this definition is also best suited for our purposes.

A \emph{knot} (more precisely, a \emph{tame knot}) is a simple closed polygonal curve in the Euclidean space~$\mathbb{R}^3$ (\emph{simple} means \emph{no self-intersections}).

The solid triangle with vertices $a$, $b$, and $c$ will be denoted by $\solidtriangle abc$.
It is defined as the convex hull of the points $a$, $b$, and $c$;
the points $a$, $b$, and $c$ are assumed to be distinct, but they might lie on one line.

We define a \emph{triangular isotopy} of a knot to be the generation of a new knot from the original one by means of the
following two operations:

\begin{itemize}
\item Assume $[p,q]$ is an edge of the knot and $x$
is a point such that the solid triangle $\solidtriangle pqx$  has no common points with the knot except for the edge $[p,q]$.
Then we can replace the edge $[p,q]$ with the two adjacent edges $[p,x]$ and $[x,q]$.
\item We can also perform the inverse operation.
That is, if for two adjacent edges $[p,x]$ and $[x,q]$ of a knot the triangle
$\solidtriangle pqx$ has no common points with the knot except for the points on the edges $[p,x]$ and $[x,q]$,
then we can replace $[p,x]$ and $[x,q]$ by one edge $[p,q]$.
\end{itemize}

Polygons that arise from one another by a finite sequence of
triangular isotopies are called \emph{isotopic}.
A knot that is not isotopic to a triangle (that is, a simple polygonal curve with three vertices) is called \emph{nontrivial}.

\begin{figure}[!ht]
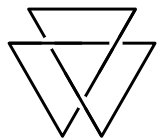

\vskip-0mm
\centering
\arXiv{\includegraphics{mppics/pic-18}}{\includegraphics{pic-18}}
\vskip0mm
\caption{Trefoil knot.}
\label{fig1}
\end{figure}

The trefoil knot shown on Figure~\ref{fig1} gives a simple example of a nontrivial knot.
A proof that it is nontrivial can be found in any textbook on knot theory.
The most elementary and visual proof is based on the so-called \emph{tricolorability} of knot diagrams \cite[Section 1.5]{adams}.

The total curvature of a smooth curve is usually defined as the integral of its curvature.
For polygons, it is defined as the sum of its external angles.
It is well known that the total curvature of a curve cannot be smaller than the total curvature of an inscribed polygonal curve (see, for example, \cite{petrunin-zamora}).
In fact, the total curvature of a curve can be defined as the least upper bound on the total curvature of inscribed polygonal curves \cite{aleksandrov-reshetnyak, sullivan-curves}.
This definition agrees with the definition given for smooth curves, and it makes sense for any simple curve.
The total curvature of a curve $\alpha$ will be denoted by $\Phi(\alpha)$.

All this leads to the following reformulation which we are going to prove.

\begin{maintheorem}
\textit{$\Phi(\alpha)\ge 4\arXiv{{\cdot}}{}\pi$ for any nontrivial knot $\alpha$.}
\end{maintheorem}

\section{Milnor--Fenner.}

One of the first solutions to the problem was found by John Milnor \cite{milnor}.
In this section, we present an amusing interpretation of his proof found by Stephen Fenner \cite{ferner}.
Just like the original version, it is based on the following sufficient condition for the triviality of a knot.

\begin{Proposition}\label{prop:one-max-one-min}
Assume that a height function $(x,y,z)\to z$ 
has only one local maximum on a simple closed polygonal curve $\alpha$ and all the vertices of the polygonal curve are at different heights.
Then $\alpha$ is a trivial knot.
\end{Proposition}

The proof is a straightforward construction of a triangular isotopy. 

\begin{proof}
Let $\alpha=p_1\dots p_n$.
We can assume that $n\ge 4$; otherwise the statement is trivial.
Note that the height function also has a unique local minimum, and $\alpha$ can be divided into two arcs from the min-vertex to the max-vertex with a monotonic height function.

Consider the three vertices with the largest height;
they have to include the max-vertex and two more.
Note that these three vertices are consecutive in the polygonal curve; 
without loss of generality, we can assume that they are $p_{n-1}$, $p_n$, and $p_1$.

Note that the solid triangle $\solidtriangle p_{n-1}p_np_1$ does not intersect any edge of $\alpha$ except the two adjacent edges $[p_{n-1},p_n]$ and $[p_n,p_1]$; see Figure~\ref{fig2}.
Indeed, if $\solidtriangle p_{n-1}p_np_1$ intersects $[p_1,p_2]$,
then, 
since $p_2$ lies below $\solidtriangle p_{n-1}p_np_1$,
the edge $[p_1,p_2]$ must intersect $[p_{n-1},p_n]$;
the latter is impossible since $\alpha$ is simple.

\arXiv{The}{In the} same way, one can show that $\solidtriangle p_{n-1}p_np_1$ cannot intersect $[p_{n-2},p_{n-1}]$.
The remaining edges lie below $\solidtriangle p_{n-1}p_np_1$, hence they cannot intersect this triangle.

Applying a triangular isotopy to $\solidtriangle p_{n-1}p_np_1$ we get a simple closed polygonal curve $\alpha'\z=p_1\dots p_{n-1}$ which is isotopic to~$\alpha$.

Since all the vertices $p_i$ have different heights,
the assumption of the proposition holds for $\alpha'$.

Repeating this procedure $n-3$ times we get a triangle.
Hence $\alpha$ is a trivial knot.
\end{proof}

\begin{figure}[!ht]
\begin{minipage}{.48\textwidth}
\centering
\arXiv{\includegraphics{mppics/pic-19}}{\includegraphics{pic-19}}
\end{minipage}\hfill
\begin{minipage}{.48\textwidth}
\centering
\arXiv{\begin{lpic}[t(-0mm),b(0mm),r(4mm),l(0mm)]{asy/Ui(1)}}{\begin{lpic}[t(-0mm),b(0mm),r(4mm),l(0mm)]{Ui(1)}}
\lbl[l]{24.5,10;$v_i$}
\lbl[lt]{22,5;$v_{i-1}$}
\lbl[tl]{18,10.5;$\phi_i$}
\lbl[br]{5,21;$U_i$}
\end{lpic}
\end{minipage}

\medskip

\begin{minipage}{.48\textwidth}
\centering
\caption{Triangular isotopy.}
\label{fig2}
\end{minipage}\hfill
\begin{minipage}{.48\textwidth}
\centering
\caption{Spherical slice.}
\label{fig3}
\end{minipage}
\vskip-0mm
\end{figure}

\begin{proof}[Proof of \arXiv{the main theorem}{Main Theorem}]
Let $\alpha=p_1\dots p_n$ be a nontrivial polygonal knot.
Denote by $v_i$ the unit vector in the direction of $p_{i+1}-p_i$;
we assume that $p_n=p_0$.

Consider the set $U_i$ formed by all unit vectors $u$ such that $\measuredangle(u,v_i)\ge \tfrac \pi 2$ and $\measuredangle(u,v_{i-1})\le \tfrac \pi 2$;
see Figure~\ref{fig3}.
Note that $u\z\in U_i$ if and only if the function $x\mapsto \langle u,x\rangle$ has a local maximum at $p_i$ on~$\alpha$; here $\langle\ ,\ \rangle$ denotes the scalar product.

Let us choose $(x,y,z)$-coordinates in the space so that the $z$-axis points in the direction of $u$.
Then according to Proposition~\ref{prop:one-max-one-min}, the function $p\mapsto \langle u,p\rangle$ has at least two local maxima on $\alpha$.
It follows that the sets $U_1,\dots,U_n$ cover each point on the unit sphere $\mathbb{S}^2$ twice.

Recall that $\area \mathbb{S}^2=4\arXiv{{\cdot}}{}\pi$.
Observe that $\phi_i=\measuredangle(v_{i-1},v_i)$ is the external angle of $\alpha$ at $p_i$.
Note that $U_i$ is a slice of the sphere between two meridians meeting at angle $\phi_i$, therefore $U_i$ occupies a $\tfrac{\phi_i}{2\arXiv{{\cdot}}{}\pi}$ portion of the whole sphere; so, $2\arXiv{{\cdot}}{} \phi_i=\area U_i$.
Since the sets $U_1, \dots, U_n$ cover $\mathbb{S}^2$ twice, we get
\[\tc\alpha=\phi_1+\dots+\phi_n=\tfrac12\arXiv{\cdot}{} (\area U_1+\dots+\area U_n)\ge \tfrac22\arXiv{\cdot}{} \area\mathbb{S}^2=4\arXiv{{\cdot}}{}\pi.\]
\end{proof}

\section{Fáry.}\label{sec:fary}

In this section, we sketch the solution of István Fáry \cite{fary} which was published before Milnor's proof.

We start with Crofton-type formulas for total curvature \cite[Proposition 4.1]{sullivan-curves}.
Given a curve $\alpha$ in $\mathbb{R}^3$ and a unit vector $u$, denote by $\alpha_{u^\perp}$ 
and $\alpha_u$ the projections of $\alpha$ to the plane perpendicular to $u$ and the line parallel to $u$, respectively.
Let us denote by $\overline{f(u)}$ the average value of $f(u)$ for a function $f\colon\mathbb{S}^2\to\mathbb{R}$.

\begin{Crofton-type formula}\label{prop:tc-crofton}
Let $\alpha$ be a polygonal curve in $\mathbb{R}^3$.
Then
\begin{align*}
\tc\alpha
&=\overline{\tc{\alpha_{u^\perp}}}
=\overline{\tc{\alpha_u}}.
\end{align*}
\end{Crofton-type formula}

\begin{proof}
Observe that it is sufficient to check the identities for $\alpha$ made of two edges.
Denote its external angle by $\phi$, so $\tc\alpha=\phi$.
Observe that each term in the formula is proportional to $\phi$.
Therefore, it is sufficient to consider the case $\phi=\pi$.

In other words, we can assume that $\alpha$ is made of two edges that turn in the opposite direction.
Note that the same holds for $\alpha_{u^\perp}$ if $u$ is not parallel to an edge of $\alpha$.
Similarly, the same holds for $\alpha_u$ if $u$ is not perpendicular to an edge of $\alpha$.
Therefore, $\tc{\alpha_{u^\perp}}=\tc{\alpha_u}=\pi$ for almost all $u\in\mathbb{S}^2$.
In particular, $\overline{\tc{\alpha_{u^\perp}}}=\overline{\tc{\alpha_u}}=\pi$ which proves the statement.
\end{proof}

The original version of Milnor's proof 
used the identity $\tc\alpha=\overline{\tc{\alpha_u}}$;
in Fenner's version of the proof, it was hidden under the rug.

Fáry's proof is based on the identity $\tc\alpha=\overline{\tc{\alpha_{u^\perp}}}$ and the following inequality for total curvature.
Suppose $\alpha=p_1\dots p_n$ is a simple closed polygonal curve in $\mathbb{R}^3$ and $o\notin\alpha$.
Let us define the \emph{angular length} of $\alpha$ with respect to $o$ as the sum
\[\Psi_o(\alpha)=\measuredangle p_{1} o p_{2}+\dots+\measuredangle p_{n-1} o p_{n}+\measuredangle p_{n} o p_{1}.\]

\begin{Proposition}\label{prop:angular-length}
For any simple closed polygonal curve and any $o\notin\alpha$, we have 
\[\Psi_o(\alpha)\le \tc{\alpha}.\]
\end{Proposition}

\begin{figure}[!ht]
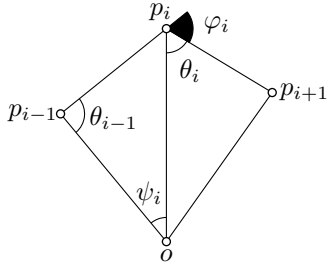

\vskip-0mm
\centering
\arXiv{\includegraphics{mppics/pic-15}}{\includegraphics{pic-15}}
\caption{Definitions of angles.}
\label{fig4}
\vskip0mm
\end{figure}

\begin{proof}
Let $\alpha=p_1\dots p_n$; for each $i$, set 
\begin{align*}
\phi_i&=\pi-\measuredangle p_{i-1}p_ip_{i+1},
&
\psi_i&=\measuredangle p_{i-1} o p_{i},
&
\theta_i&=\measuredangle o p_i p_{i+1};
\end{align*}
see Figure~\ref{fig4}.
Here we assume that the indices are taken modulo $n$; in particular, $p_{n}\z=p_0$.

Note that $\phi_i$ is the external angle at $p_i$;
therefore 
\[\tc\alpha= \phi_1+\dots+\phi_n.\]
The directions of $p_i-p_{i-1}$, $o-p_i$, and $p_{i+1}-p_i$ make angles 
$\psi_i+\theta_{i-1}$, $\theta_i$, and $\phi_i$ to each other.
Applying the triangle inequality for these angles, we get
\[\phi_i\ge \psi_i+\theta_{i-1}-\theta_i.\]
Summing up, we get
\[\phi_1+\dots+\phi_n\ge \psi_1+\dots+\psi_n,\]
and the result follows.
\end{proof}

\begin{proof}[Proof of \arXiv{the main theorem}{Main Theorem}]
Consider a projection $\alpha_{u^\perp}$ of the knot $\alpha$ to a plane in \emph{general position}
(\arXiv{this time it means}{meaning} that the self-intersections of the projection are at most double points and the projection of each edge is not degenerate).
The closed polygonal curve $\alpha_{u^\perp}$ divides the plane into domains, one of which is unbounded, denoted by $U$, and the others are bounded.

\begin{figure}[!ht]
\begin{minipage}{.48\textwidth}
\centering
\arXiv{\includegraphics{mppics/pic-13}}{\includegraphics{pic-13}}
\end{minipage}\hfill
\begin{minipage}{.48\textwidth}
\centering
\arXiv{\includegraphics{mppics/pic-14}}{\includegraphics{pic-14}}
\end{minipage}

\medskip

\begin{minipage}{.48\textwidth}
\centering
\caption{Projection of an unknot.}
\label{fig5}
\end{minipage}\hfill
\begin{minipage}{.48\textwidth}
\centering
\caption{Projection of a knot.}
\label{fig6}
\end{minipage}
\vskip-0mm
\end{figure}

First, note that all domains can be colored in a chessboard order;
that is, they can be colored in black and white in such a way that domains with common borderline get different colors \cite[Exercise 2.27]{adams}.
If the unbounded domain is colored white and every other domain is colored black (see Figure~\ref{fig5}), then one can untie the knot by flipping these domains one by one.%
\footnote{It is instructive to give a formal proof of the last statement; that is, \textit{show that if there is only one white region, then $\alpha$ is trivial}.}

Therefore, there is a white bounded domain; denote it by $D$ (see Figure~\ref{fig6}).
The domain $D$ cannot adjoin 
$U$, since they have the same color.
Fix a point $o$ in this domain.

Since any ray from $o$ crosses $\alpha_{u^\perp}$ twice, we get $\Psi_o(\alpha_{u^\perp})\ge 4\arXiv{{\cdot}}{}\pi$;
that is, the angular length of $\alpha_{u^\perp}$ with respect to $o$ is at least $4\arXiv{{\cdot}}{}\pi$. 
By Proposition~\ref{prop:angular-length}, we have 
\[\tc{\alpha_{u^\perp}}\ge4\arXiv{{\cdot}}{}\pi.\]
This is true for any $u$ in general position.
The remaining directions contribute nothing to the average value.
It remains to apply the Crofton-type formula $\tc{\alpha}=\overline{\tc{\alpha_{u^\perp}}}$.
\end{proof}

\section{Pannwitz--Hopf--Schmitz--Denne.}\label{sec:quadrisecant}

Fáry's paper \cite{fary} is ended with the following note:
``I just received a letter from Mr. Borsuk, which says that Theorem 3 (\textit{our \arXiv{main theorem}{Main Theorem}}) has been proved independently by Mr. H. Hopf.
It uses the theorem of Miss Pannwitz, which ensures that for any knot there is a line that crosses it at least in four
points.''
%Je viens de recevoir une lettre de M. Borsuk, qui m'apprend que le théorème 3 a été démontré indépendamment par M. H. Hopf. Il utilise le théorème de Mlle Pannwitz, qui assure que pour tout nœud on peut trouver une droite le coupant au moins en quatre points.
This proof is the subject of this section.

The main step in the proof is the existence of the so-called \emph{alternated quadrisecant} of a knot $\alpha$; that is, there is a line $\ell$ that shares with $\alpha$ four points $a$, $b$, $c$, and $d$
that appear on $\ell$ in the same order and have cyclic order $a$, $c$, $b$, $d$ on $\alpha$.
See Figure~\ref{fig7}.

\begin{figure}[!ht]
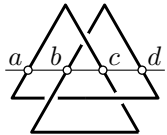

\vskip-0mm
\centering
\arXiv{\includegraphics{mppics/pic-41}}{\includegraphics{pic-41}}
\caption{Alternated quadrisecant.}
\label{fig7}
\vskip0mm
\end{figure}

The question about the existence of a quadrisecant was posed by Otto Toep\-litz and answered by Erika Pannwitz for tame knots in \emph{general position} \cite{pannwitz}.
She only proved the existence of a quadrisecant, but the existence of an alternated quadrisecant can be extracted from her proof.
Elizabeth Denne had generalized this result to all knots \cite{denne, denne-survey}.
But for us tame knots in general position will be sufficient;
namely, we need the following.

\begin{Proposition}\label{prop:quadrisecant}
Any nontrivial knot $\alpha$ in general position admits an \emph{alternated quadrisecant}.
\end{Proposition}

The precise meaning of general position will be clear from the proof;
what is important is that any knot is arbitrarily close to a knot in general position.

\begin{proof}[Proof of \arXiv{the main theorem}{Main Theorem} modulo Proposition~\ref{prop:quadrisecant}]
Let $\alpha=p_1\dots p_n$ be a knot in general position.
Suppose $a$, $b$, $c$, and $d$ be the points as in the definition of an alternated quadrisecant.
Then $acbd$ is an inscribed quadrangle with all external angles equal to~$\pi$.
Therefore, $\tc{abcd}=4\arXiv{{\cdot}}{}\pi$.

Since the total curvature of an inscribed polygonal curve cannot be larger than the total curvature of the original curve, we get that $\tc{\alpha}\z\ge 4\arXiv{{\cdot}}{}\pi$.

Finally, for any knot $\beta=q_1\dots q_n$ there is an arbitrarily close knot $\alpha\z=p_1\dots p_n$ in general position;
in particular, for any $\eps>0$ we can assume that $\tc{\beta}>\tc\alpha-\eps$.
It follows that $\tc{\beta}>4\arXiv{{\cdot}}{}\pi-\eps$ for any positive $\eps$ --- hence the result.
\end{proof}

In the proof of the proposition, we will use the following characterization of the trivial knot.

\begin{Lemma}\label{lem5}
A knot $\alpha$ is trivial if there exists a piecewise linear map $F$ from the disc $\DD$ to $\RR^3$ such that the  restriction of $F$ to the boundary $\partial\DD=\mathbb{S}^1$ is a degree-one map to $\alpha$ and $F$ does not map interior points of $\DD$ to $\alpha$.
\end{Lemma}

This statement can be deduced from  the \textit{loop theorem} --- a heavy weapon of 3-dimensional topology.
For those who are familiar with the loop theorem, it would be an exercise; otherwise, we suggest taking it for granted.
In the following proof, we follow closely the presentation of Carsten Schmitz \cite{schmitz} and the original argument of Erika Pannwitz.

\begin{proof}[Proof of Proposition~\ref{prop:quadrisecant}]
We may assume that $\alpha$ comes with a 1-periodic piecewise linear parametrization by $\RR$;
so the space of oriented chords of $\alpha$ can be identified with the open cylinder $\mathbb{S}^1\times (0,1)$, where $\mathbb{S}^1=\RR/\ZZ$.
Namely, we assume that a pair $(x,y)\in \mathbb{S}^1\times (0,1)$ corresponds to the oriented chord with the ends at $\alpha(x)$ and $\alpha(x+y)$.

Choose a pair $(x,y)\in \mathbb{S}^1\times (0,1)$.
Let us denote by $r(x,y)$ the ray that starts at $\alpha(x)$ and goes in the direction opposite to $\alpha(x+y)$; see Figure~\ref{fig8}.
We write $(x,y)\z\in C_3$ if $r(x,y)$ crosses $\alpha$ at another point.

The points $\alpha(x)$ and $\alpha(x+y)$ divide the knot into two open arcs $\alpha|_{(x,x+y)}$ and $\alpha|_{(x+y,x+1)}$.
If $(x,y)\z\in C_3$ and $r(x,y)$ crosses the second arc, then we write $(x,y)\in C_3^+$;
if it crosses the first arc, then $(x,y)\z\in C_3^-$.
Note that $C_3^+\cup C_3^-\z=C_3$.

\begin{figure}[!ht]
\begin{minipage}{.48\textwidth}
\centering
\arXiv{\includegraphics{mppics/pic-43}}{\includegraphics{pic-43}}
\end{minipage}\hfill
\begin{minipage}{.48\textwidth}
\centering
\arXiv{\includegraphics{mppics/pic-46}}{\includegraphics{pic-46}}
\end{minipage}

\medskip

\begin{minipage}{.48\textwidth}
\centering
\caption{Ray $r(x,y)$.}
\label{fig8}
\end{minipage}\hfill
\begin{minipage}{.48\textwidth}
\centering
\caption{Colored set.}
\label{fig9}
\end{minipage}
\vskip-0mm
\end{figure}

Observe that if $C_3^+$ and $C_3^-$ intersect, then the proposition follows.
In general, the set $C_3^\pm$ is not closed;
denote by $\bar C_3^\pm$ its closure.
Suppose 
\[\bar C_3^+\cap \bar C_3^-=\emptyset.\]
Note that in this case there is a large $n$ such that any $\tfrac1n\times\tfrac1n$-square in $\mathbb{S}^1\times (0,1)$ does not intersect both $\bar C_3^+$ and $\bar C_3^-$.
Let us cut $\mathbb{S}^1\times (0,1)$ into $n^2$ such squares; each square is a closed subset of $\mathbb{S}^1\times (0,1)$.
Color the union of squares that intersect $\bar C_3^+$; see Figure~\ref{fig9}.
Note that every square in the lowest row is colored and that we did not color squares in the upper row.
Further, the boundary of the colored set contains a simple curve $t\mapsto(x(t),y(t))$ that cuts  the cylinder $\mathbb{S}^1\times (0,1)$ into two cylinders.

Note that $(x(t),y(t))\notin \bar C_3^\pm$ for any $t\in \mathbb{S}^1$ and the curve $t\mapsto(x(t),y(t))$ runs along coordinate lines.
Consider the one-parameter family of line segments in $r(x(t),y(t))$ that start at $\alpha(x(t))$ and end on the surface of a large cube that contains $\alpha$ in its interior.
In this way we obtain a piecewise linear annulus that connects the curve $t\mapsto \alpha(x(t))$ to a curve on the surface of the cube.
The latter curve can be contracted by a piecewise linear disc in the surface of the cube.
It might have self-intersection, but it cannot contains points of $\alpha$.

Observe that $t\mapsto \alpha(x(t))$ defines a degree-one map $\mathbb{S}^1\z\to\alpha$.
Applying Lemma~\ref{lem5}, we get the result.

It remains to show that if $\alpha=p_1\dots p_n$ is in general position, then
\[C_3^+\cap C_3^-=\emptyset\quad\text{implies that}\quad
\bar C_3^+\cap \bar C_3^-=\emptyset.\]
Assume $C_3^+\cap C_3^-=\emptyset$ and $(x,y)\in \bar C_3^+\cap \bar C_3^-$.
Denote by $\ell$ the line containing $\alpha(x)$ and $\alpha(x+y)$.
Checking the following statements is straightforward, but requires patience:
\begin{itemize}
 \item $\alpha(x)$ is a vertex, so $\alpha(x)=p_i$ for some $i$ and $\ell$ contains $p_i$;
 \item $\ell$ lies in the plane spanned by $p_{i-1}$, $p_i$, and $p_{i+1}$; 
 \item $\ell$ does not contain edge $[p_{i-1}, p_i]$, nor $[p_{i}, p_{i+1}]$;
 \item $\ell$ has at least 3 points of intersection with $\alpha$.
\end{itemize}
Finally, if $\alpha$ is in general position,
the line $\ell$ with the described properties does not exist.
\end{proof}

\section{Alexander--Bishop.}

Here we sketch the proof given by Stephanie Alexander and Richard Bishop \cite{alexander-bishop}.
This proof was designed to work for more general ambient spaces.
As a result, it is more elementary.

In the proof, we construct a total-curvature-decreasing deformation of a given knot into a doubly covered
bigon.
The statement follows since the latter has total curvature $4\arXiv{{\cdot}}{}\pi$.

\begin{proof}[Proof of \arXiv{the main theorem}{Main Theorem}.]
Let $\alpha=p_1\dots p_n$ be a nontrivial knot;
that is, one cannot get a triangle from $\alpha$ by applying a sequence of triangular isotopies defined in the introduction.

If $n=3$ the polygonal curve $\alpha$ is a triangle.
Therefore, by definition, $\alpha$ is a trivial knot --- there is nothing to show.

Consider the smallest $n$ for which the statement fails;
that is, there is a nontrivial knot $\alpha\z=p_1\dots p_n$ such that
\[\tc\alpha<4\arXiv{{\cdot}}{}\pi.
\arXiv{\leqno({*})}{\eqno(1)}\]
We use the indices modulo $n$; that is, $p_0=p_n$, $p_1=p_{n+1}$, and so on.
Without loss of generality, we may assume that $\alpha$ is in \emph{general position}; 
this time it means that no four vertices of $\alpha$ lie on one plane. 

Set $\alpha_0=\alpha$.
If the solid triangle $\solidtriangle p_{0}p_1p_{2}$ intersects $\alpha_0$ only in the two adjacent edges,
then applying the corresponding triangular isotopy, we get a knot $\alpha'_0$ with $n-1$ edges that is inscribed in $\alpha_0$.
Therefore,
\[\tc{\alpha_0}\ge \tc{\alpha_0'}.\]
On the other hand, by minimality of $n$, 
\[\tc{\alpha_0'}\ge 4\arXiv{{\cdot}}{}\pi,\]
which contradicts $\arXiv{({*})}{(1)}$.

\begin{figure}[!ht]
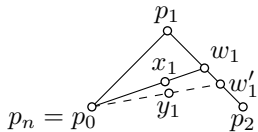

\vskip-0mm
\centering
\arXiv{\includegraphics{mppics/pic-17}}{\includegraphics{pic-17}}
\caption{Construction of $w_1$.}
\label{fig10}
\vskip0mm
\end{figure}

Let $w'_1$ be the first point on the edge $[p_1,p_2]$ such that the line segment $[p_0,w'_1]$ 
intersects $\alpha_0$, say at $y_1$. %???

Choose a point $w_1$ on $[p_1,p_2]$ a bit before $w'_1$.
Denote by $x_1$ the point on $[p_0,w_1]$ that minimizes the distance to $y_1$.
In this way we get a closed polygonal curve 
$\alpha_1\z=w_1p_2\dots p_n$ with two marked points $x_1$ and $y_1$; see Figure~\ref{fig10}.
Denote by $m_1$ the number of edges in the arc $x_1w_1\dots y_1$ of $\alpha_1$.

Note that $\tc{\alpha_0}\ge \tc{\alpha_1}$, and
$\alpha_1$ is isotopic to $\alpha_0$;
in particular, $\alpha_1$ is a nontrivial knot.
Moreover, the point $w_1$ can be chosen so that $\alpha_1$ is in general position.

Now let us repeat the procedure for the adjacent edges $[w_1,p_2]$ and $[p_2,p_3]$ of $\alpha_1$.
If the solid triangle $\solidtriangle w_1p_2p_3$ intersects $\alpha_1$ only at these two adjacent edges, then we get a contradiction the same way as before.
Otherwise, we get a new knot $\alpha_2\z=w_1w_2p_3\dots p_n$ with two more marked points $x_2$ and $y_2$.
Denote by $m_2$ the number of edges in the polygonal curve $x_2w_2\dots y_2$.

Note that the points $x_1,x_2,y_1,y_2$ cannot appear on $\alpha_2$ in the same cyclic order;
otherwise the polygonal curve $x_1x_2y_1y_2$ can be made to be arbitrarily close to a doubly covered bigon which again contradicts~$\arXiv{({*})}{(1)}$.

Therefore, we can assume that the arc $x_2w_2\dots y_2$ lies inside the arc $x_1w_1\dots y_1$ in $\alpha_2$
and therefore $m_1>m_2$.

Continuing this procedure we get a sequence of polygonal curves $\alpha_i\z=w_1\z\dots w_i p_{i+1}\z\dots p_n$ with marked points $x_i$ and $y_i$ such that the number of edges $m_i$ from $x_i$ to $y_i$ decreases as $i$ increases.
Clearly $m_i>1$ for any $i$ and $m_1<n$.
Therefore, it requires fewer than $n$ steps to arrive at a contradiction.
\end{proof}

\section{Ekholm--White--Wienholtz.}

In this section, we discuss a solution of the problem based on the theorem of Tobias Ekholm, Brian White, and Daniel Wienholtz \cite{EWW_embed}.
This theorem was a breakthrough in minimal surface theory at the time.
Yet it was based on an elementary idea that we are going to explain.

We start with a polygonal curve $\alpha$ with total curvature less than $4\arXiv{{\cdot}}{}\pi$;
show that an area-minimizing disc spanned by $\alpha$ has no self-intersections, and therefore $\alpha$ has to be a trivial knot.
So in a way the equation for area-minimizing surfaces solves our problem; we only need to understand~it.

The main hero in this proof is the so-called extended monotonicity theorem.
We will also apply the Douglas--Rado theorem on the existence of area-minimizing discs and reuse the inequality between total curvature and angular length from Fáry's proof; see Proposition~\ref{prop:angular-length}.

The image of a map from a domain of $\mathbb{R}^2$ to $\mathbb{R}^3$ will be called a \emph{surface};
it might have self-intersections and singularities, but we assume it is reasonable, say locally Lipschitz \cite{wiki:lipschitz}; so we can talk about its area.
A point on the surface might refer to a point in $\mathbb{R}^3$, or to the corresponding point in the domain of parameters in $\mathbb{R}^2$;
it should be easy to infer from the context.

We denote by $\mathbb{D}$ the closed disc in the plane.
A surface defined by a map $f\colon\mathbb{D}\to\mathbb{R}^3$ will be called a \emph{disc}.
The restriction $f|_{\partial \mathbb{D}}$ as well as its image can be referred to as the \emph{boundary} of the disc.

A disc $\Sigma$ is called area-minimizing if it has the smallest area among the discs with the given boundary.
The following statement about area-minimizing discs is easy to believe, but not easy to prove; see \cite{white-lectures}.

\begin{theorem}[Douglas--Rado]\label{thm:min-exists} 
Given a simple closed polygonal curve $\alpha$ in $\mathbb{R}^3$, there is an area-minimizing disc $\Sigma$ with boundary $\alpha$; it is a smooth surface, possibly with self-intersections and isolated singularities.
Moreover, if $\Sigma$ has no self-intersections, then it is an embedded smooth surface with no singularities (in this case $\alpha$ is a trivial knot).
\end{theorem}

Choose a disc $\Sigma$ in $\mathbb{R}^3$ with boundary $\alpha$.
Given a point $o\notin \alpha$, let us consider the \emph{collared} $\Sigma$ with respect to~$o$;
it is a new surface that will be denoted by $\hat\Sigma_o$;
it includes $\Sigma$ and the \emph{collar} formed by all rays that start at points of $\alpha$ and go in the direction opposite to $o$; see Figure~\ref{fig11}.
Note that $\hat\Sigma_o$ admits a natural parametrization by the whole plane.

\begin{figure}[!ht]
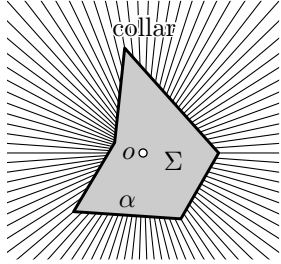

\vskip-0mm
\centering
\arXiv{\includegraphics{mppics/pic-51}}{\includegraphics{pic-51}}
\caption{Surface and its collar.}
\label{fig11}
\vskip0mm
\end{figure}

\arXiv{\begin{theorem}[extended monotonicity]}{\begin{theorem}[Extended Monotonicity]}
\label{thm:monotonicity}
Let $\Sigma$ be an area-minimizing disc with boundary~$\alpha$.
Given a point $o\notin \alpha$, consider the function 
\[W_o(r)=\area (\hat\Sigma_o\cap \bar B_r(o)),\]
where $\bar B_r(o)$ denotes the ball of radius $r$ centered at $o$.
Then $r\z\mapsto \frac{W_o(r)}{r^2}$ is a nondecreasing function.
Moreover, 

\noindent(a)
$\lim_{r\to\infty}\frac{W_o(r)}{r^2}=\tfrac12\arXiv{\cdot}{} \Psi_o(\alpha)$, where $\Psi_o(\alpha)$ denotes the angular length of $\alpha$ with respect to $o$; see Section~\ref{sec:fary};

\noindent(b) if $o\in \Sigma$, then $\lim_{r\to0}\frac{W_o(r)}{r^2}\ge \pi$.

\end{theorem}

The classical monotonicity theorem states that the function $r\mapsto \frac{W_o(r)}{r^2}$ is monotonic 
while $r$ is smaller than the distance from $o$ to $\alpha$.
The stated version is due to Brian White \cite{white}.
The same statement holds for minimal surfaces \cite{EWW_embed}; its proof requires a deeper dive into differential geometry.
At the same time, the original formulation admits a generalization to a wider class of ambient spaces \cite{St_structure}.

\begin{proof}
Denote by $\lambda_o(r)$ the curve of intersection of the sphere $\partial B_r(o)$ with $\hat\Sigma_o$;
set $\ell(r)\z=\length[\lambda_o(r)]$.
Observe that 
\[W_o'(r)\ge \ell(r)\]
for almost all $r$.
(Formally speaking, this inequality follows from the so-called \emph{coarea formula}.)

Set $\Delta_r=\hat\Sigma_o\cap \bar B_r(o)$;
it is a surface bounded by $\lambda_o(r)$.
Let $\tilde\Delta_r$ be the cone over $\lambda_o(r)$ with the center at $o$.
Note that $\tilde\Delta_r$ differs from $\Delta_r$ only inside $\Sigma$.
Since $\Sigma$ is area-minimizing, we get that 
\[\area \tilde\Delta_r\ge \area \Delta_r
\arXiv{\leqno({*}{*})}{\eqno(2)}\]
for any $r>0$.
Observe that 
\begin{align*}
\area \Delta_r&=W_o(r),
&
\area \tilde\Delta_r&=\tfrac12\arXiv{\cdot}{} r\arXiv{\cdot}{} \ell(r).
\end{align*}
Applying $\arXiv{({*}{*})}{(2)}$, we get
\[r\arXiv{\cdot}{} \ell(r)\ge 2\arXiv{\cdot}{} W_o(r).\]
Therefore, 
\[r\arXiv{\cdot}{} W_o'(r)\ge 2\arXiv{\cdot}{} W_o(r)\]
for almost all $r$.
If $W_o$ is smooth, then this inequality implies the main statement.
In general, $(\frac{W_o(r)}{r^2})'\ge 0$ holds almost everywhere, and this is enough to conclude the monotonicity of $r\z\mapsto \frac{W_o(r)}{r^2}$ since $W_o$ is nondecreasing.

\noindent\textit{(a).}
Observe that up to a fixed error we have that $W_o(r)$ is the area of the ball of radius $r$ in the cone over $\alpha$ with the tip at $o$.
It follows that $\frac{W_o(r)}{r^2}$ approaches the area of the unit ball in this cone as $r\to\infty$ --- hence the result.

\noindent\textit{(b).}
The statement is evident for smooth points of $\Sigma$.
Since smooth points are dense in $\Sigma$, and $o\mapsto W_o(r)$ is a continuous function,
the main part of the theorem implies that $W_o(r)\ge\pi\arXiv{\cdot}{} r^2$ for \emph{any} point $o\in\Sigma$ --- hence the result.
\end{proof}

\begin{proof}[Proof of \arXiv{the main theorem}{Main Theorem}]
Suppose that $\tc\alpha<4\arXiv{{\cdot}}{}\pi$.
Consider an area-minimizing surface $\Sigma$ with the boundary $\alpha$; it exists by \arXiv{the Douglas–Rado theorem}{Theorem~\ref{thm:min-exists}}.
If $\Sigma$ has no self-intersections, then $\alpha$ is a trivial knot.

Suppose $\Sigma$ has a self-intersection at a point $o$.
In this case, the intersection $B_r(o)\z\cap \Sigma$ is covered by two or more small area-minimizing subdiscs of $\Sigma$.
By Theorem~\ref{thm:monotonicity}(b), we get 
\[\lim_{r\to0}\frac{W_o(r)}{r^2}\ge 2\arXiv{{\cdot}}{}\pi.\]

Applying the main statement in \arXiv{the monotonicity theorem}{Theorem~\ref{thm:monotonicity}}, we get $\frac{W_o(r)}{r^2}\z\ge 2\arXiv{{\cdot}}{}\pi$ for any $r>0$.
By Theorem~\ref{thm:monotonicity}(b) and Proposition~\ref{prop:angular-length} in Fáry's proof, we get
\[\tc{\alpha}\ge \Psi_o(\alpha)\ge 2\arXiv{\cdot}{} \frac{W_o(r)}{r^2}\ge 4\arXiv{{\cdot}}{}\pi\]
--- a contradiction.
\end{proof}

\section{Cantarella--Kuperberg--Kusner--Sullivan.}\label{sec:2nd-hull}

The following proof is due to Jason Cantarella, Greg Kuperberg, Robert Kusner, and John Sullivan \cite{CKKS};
this is the only proof in our collection that uses knot theory a bit beyond the basic definitions.

\begin{figure}[!ht]
\vskip-0mm
\centering
\arXiv{\includegraphics{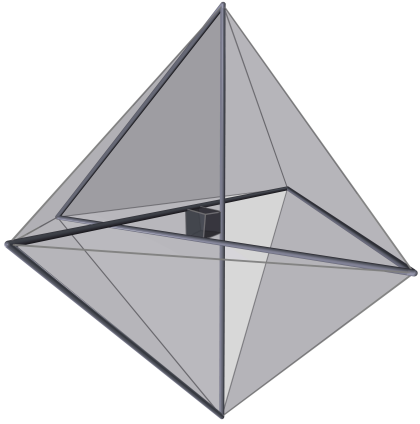}}{\includegraphics{trefoil}}
\caption{First and second hulls of a trefoil.}
\label{fig12}
\vskip0mm
\end{figure}

Choose a polygonal curve $\alpha\z=p_1\dots p_n$.
Suppose a plane $\Pi$ is in general position; that is, it does not contain vertices of $\alpha$.
Let us denote by $\cross_\alpha(\Pi)$ the number of intersections of $\Pi$ and $\alpha$.
Extend the function $\Pi\z\mapsto \cross_\alpha(\Pi)$ to the minimal upper semicontinuous function defined for all planes;
in other words, $\cross_\alpha(\Pi)$ is the maximal integer $k$ such that there is a generic plane $\Pi'$ arbitrarily close to $\Pi$ that intersects $\alpha$ at $k$ points.
The number $\cross_\alpha(\Pi)$ will be called the \emph{crossing number} of~$\Pi$.
Note that the crossing number is always even, and it cannot exceed $n$.

It is easy to see that the \emph{convex hull} $h_1(\alpha)$ of $\alpha$ can be defined in the following way:
\textit{$x\in h_1(\alpha)$ if $\cross_\alpha(\Pi)\z\ge2$ for any plane $\Pi$ containing $x$}.
This observation suggests the following definition of the \emph{second hull}:
\textit{$x\in h_2(\alpha)$ if $\cross_\alpha(\Pi)\ge4$ for any plane $\Pi$ containing $x$}.
See Figure~\ref{fig12}.

\begin{Theorem}\label{thm:2nd-hull}
The second hull of any nontrivial knot $\alpha$ has a nonempty interior.
\end{Theorem}

Recall that \emph{spherical polygonal curve} is defined as a concatenation of a finite number of great-circle arcs on the unit sphere.
To show that Theorem~\ref{thm:2nd-hull} implies \arXiv{the main theorem}{Main Theorem}, we will apply the spherical Crofton formula:
\textit{for any spherical polygonal curve $\gamma$ we have}
\[\length \gamma=\pi\arXiv{{\cdot}}{} \overline n,\arXiv{\leqno(\threestars)}{\eqno(3)}\]
\textit{where $\overline n$ denotes the average number of intersections of $\gamma$ with equators.}
To prove this formula, check it for an arc and sum it up for all edges of $\gamma$.

\begin{proof}[Proof of \arXiv{the main theorem}{Main Theorem} modulo Theorem~\ref{thm:2nd-hull}]
Choose a point $o\in h_2(\alpha)$; we can assume that $o\notin\alpha$.
Consider the radial projection $\alpha^*$ of $\alpha$ to the unit sphere centered at $o$;
observe that 
\[\length\alpha^*=\Psi_o(\alpha),\]
where $\Psi_o(\alpha)$ denotes the angular length of $\alpha$ with respect to $o$; see Section~\ref{sec:fary}.

By Proposition~\ref{prop:angular-length}, it is sufficient to show that 
\[\length\alpha^*\ge 4\arXiv{{\cdot}}{}\pi.
\arXiv{\leqno(\fourstars)}{\eqno(4)}\]
Since $o$ is in the second hull, $\alpha^*$ crosses every equator in general position at least 4 times.
It follows that the average number of crossings is at least~4.
Applying $\arXiv{(\threestars)}{(3)}$, we get $\arXiv{(\fourstars)}{(4)}$.
\end{proof}

Suppose that a plane $\Pi$ divides a knot $\alpha$ into two arcs, one on each side; in particular, $\Pi$ intersects $\alpha$ at two points, say $p$ and $q$.
Then we can create two knots $\alpha_1$ and $\alpha_2$ by joining the ends of the two arcs by the line segment $[p,q]$.
In this case, we say that $\alpha$ is a \emph{connected sum} of $\alpha_1$ and $\alpha_2$.

\begin{Claim}\label{clm:connected-sum}
Suppose that a knot $\alpha$ is a connected sum of knots $\alpha_1$ and $\alpha_2$.
If at least one of the knots $\alpha_1$ or $\alpha_2$ is nontrivial, then so is $\alpha$.
\end{Claim}

This claim has an amusing proof via the so-called \emph{infinite swindle} \cite{mazur}; see also \cite{poenaru}.
It also follows by additivity of knot genus \cite[Section 4.3]{adams}.

Suppose $\beta$ is a polygonal curve \emph{inscribed} in $\alpha$;
that is, the vertices of $\beta$ lie on $\alpha$ and they appear in the same cyclic order on $\alpha$ and $\beta$.
If a plane in general position intersects an edge of $\beta$, then it intersects
the corresponding arc of $\alpha$. %taking orientation into account we can say arc,
%arcs would also make sense
Therefore, we get the following.

\begin{Observation}
If a polygonal line $\beta$ is inscribed in $\alpha$, then 
\[h_2(\beta)\z\subset h_2(\alpha).\]
\end{Observation}

\begin{proof}[Proof of Theorem~\ref{thm:2nd-hull}]
Assume the contrary; let $\alpha$ be a nontrivial polygonal knot with the smallest number of vertices, say $n$, such that $h_2(\alpha)$ has an empty interior.
It is easy to see that $n\ge 6$;
in fact, any simple space 5-gon is a trivial knot.

Suppose $\Pi$ is a plane in general position that divides $\alpha$ into two arcs; thus it defines a decomposition of $\alpha$ into a connected sum of two knots $\alpha_1$ and $\alpha_2$.
By the observation, $h_2(\alpha_1)$ and $h_2(\alpha_2)$ have empty interiors.
It follows that one of these knots, say $\alpha_1$, is trivial;
therefore, the other, respectively $\alpha_2$, is isotopic to $\alpha$.
Indeed, if $n_1$ and $n_2$ denote the number of vertices in $\alpha_1$ and $\alpha_2$, then $n_1+n_2=n+4$. 
If both knots are nontrivial, then $n_1\ge 6$ and $n_2\ge 6$.
Therefore, $n_1<n$ and $n_2<n$, which contradicts the minimality of $n$.

The open half-space $H$ bounded by $\Pi$ and containing $\alpha_2\setminus\Pi$ will be called \emph{essential}. Intuitively, an essential half-space cuts off
a trivial knot of $\alpha$. (A half-space containing all of $\alpha$ will be considered essential as well.)
A rather straightforward application of Claim~\ref{clm:connected-sum} implies that if $H'$ is another essential half-space for $\alpha$, then it is also essential for $\alpha_2$; see Figure~\ref{fig13}.
It follows that the intersection of all essential half-spaces, say $W$, has nonempty interior --- roughly speaking, it has to contain the region where the knotting of $\alpha$ takes place.

\begin{figure}[!ht]
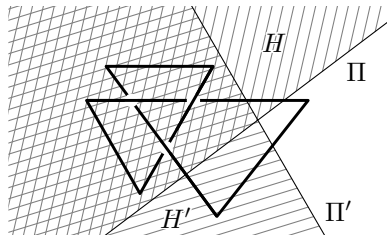

\vskip-0mm
\centering
\arXiv{\includegraphics{mppics/pic-55}}{\includegraphics{pic-55}}
\vskip0mm
\caption{Two essential half-spaces.}
\label{fig13}
\end{figure}

Finally observe that if $\cross_\alpha(\Pi)=2$ for a plane $\Pi$ in general position, then $\Pi$ bounds an essential half-space. 
It follows that if a plane in general position intersects $W$, then it has  crossing number at least 4,
so $h_2(\alpha)\supset W$ --- hence the result.
\end{proof}

\begin{acknowledgment}{Acknowledgment.}
We wish to thank anonymous referees for their thoughtful reading and insightful suggestions.
The first author was partially supported by the NSF grant DMS-2005279, the Simons Foundation grant \#584781, and Minobrnauki, grant \#075-15-2022-289.
The second author was partially supported by DFG grant SPP 2026. 
\end{acknowledgment}

\end{document}